\documentclass{amsart}
\usepackage{amsthm}
\usepackage{url}
\usepackage{amsmath}
\usepackage{amsfonts}

\usepackage{amssymb}
\newtheorem{theorem}{Theorem}
\newtheorem{proposition}[theorem]{Proposition}
\newtheorem{lemma}[theorem]{Lemma}

\newtheorem{definition}[theorem]{Definition}

\begin{document}
\title{The Coefficient-Choosing Game}
\author{William Gasarch}
\address{Department of Computer Science,
University of Maryland,
College Park, MD 20742}
\email{gasarch@cs.umd.edu}
\author{Lawrence C. Washington}
\address{Department of Mathematics,
University of Maryland,
College Park, MD 20742}
\email{lcw@math.umd.edu}
\author{Sam Zbarsky}
\address{Department of Mathematics,
Princeton University,
Princeton, NJ 08544}
\email{zbarskysam@gmail.com}

\date{}

\maketitle

\begin{abstract}
Let $D$ be an integral domain. 
Two players, Nora and Wanda, alternately choose coefficients from $D$ for a polynomial
of degree $d$. When they are done, if the polynomial 
has a root in the field of fractions of $D$, then Wanda wins.
If not, then Nora wins.
We determine, for  many $D$, who wins this game.
\end{abstract}

\section{Introduction}

Wanda and Nora are choosing the coefficients of a polynomial 
$$a_3x^3+a_2x^2+a_1x+a_0$$ of degree 3 with integer coefficients.
Wanda wants the final polynomial to have a rational root, and Nora wants the polynomial to have no rational
roots. Wanda starts by choosing $a_2=-12$. Nora responds with $a_3=7$. Wanda then chooses $a_0=4$
(she is not allowed to choose $a_0=0$). It remains for Nora to choose $a_1$ so that
$$
7x^3-12x^2+a_1x+4$$
has no rational root. Fortunately, Nora knows the Rational Root Theorem, which implies that
if she chooses $a_1$ to be an integer then the only possible rational roots of the polynomial are 
$$
\pm1/7, \; \pm 2/7, \; \pm 4/7,\; \pm 1, \; \pm 2, \; \pm 4.
$$
 By choosing $a_1= 10000$, she avoids these roots and thereby wins the game.

This is a simple case of the Coefficient-Choosing Game. In the following, we describe the game and give
winning strategies, depending on what ring is allowed for the coefficients of the polynomial.
We start with the simple example of subrings of the rationals, which relies on unique
factorization and the Rational Root Theorem. When we try to extend the proofs to coefficients lying in finite
extensions of $\mathbb Q$, we need to use some algebraic number theory to handle the possibility
of non-unique factorization. But we also meet a deep result from transcendental number theory concerning the $S$-unit
equation, which is an important tool in Diophantine equations.

Section \ref{se:R} treats the case $D=\mathbb R$ and Section \ref{se:F}  
considers the case  $D$ is a finite field, where an unexpected
special case arises in degree 3 that is related to permutation polynomials.

\section
{The Game}\label{game}

The parameters of the game are an integral domain $D$ and a degree $d\ge 2$ (the case $d=1$ is trivial).
Either Wanda (``wants root'') or Nora (``no root'') is chosen to be player I, and the other becomes player II.
Here are the rules:
\begin{enumerate}
\item Player I goes first. 
\item Players I and II alternately choose coefficients from $D$ for
a polynomial of degree $d$. The coefficients are
not chosen in any pre-determined order.
\item They must choose $a_d\ne 0$ and $a_0\ne 0$ (in order to avoid trivial situations).
\item If the final polynomial has a root in the field of fractions of $D$, then Wanda wins, otherwise Nora wins.
\end{enumerate}

We assume both players play perfectly. The statement
{\it Player I (II) wins} means that Player $I$ ($II$) has a strategy that wins no matter
what the other player does.

To get a feel for the game, the reader might want to try playing the final move
in the following situations, both as Nora
and as Wanda:
\begin{enumerate}
\item $D=\mathbb Q$ and $a_0$ remains yet to be chosen: $\frac1{6}x^3+24x^2-\frac14x+a_0$.
\item $D=\mathbb R$ and $a_1$ remains to be chosen: $5x^4-6x^3-3x^2+a_1x+10$.
\item $D=\mathbb R$ and $a_1$ remains to be chosen: $x^3+3x^2+a_1x-5$.
\item $D=\mathbb Z[\sqrt{2}]$ and $a_1$ remains to be chosen: 
$x^3+(\sqrt{2}-3)x^2+a_1x-4(1+\sqrt{2})$.
\end{enumerate}
All of these can be won by Wanda if she is the one playing. If Nora plays,
the first two examples can easily be won by her after some numerical experimentation. 
The third example is a win for Wanda, no matter what Nora plays. 
However, the last example might not be as easy. The ring $\mathbb Z[\sqrt{2}]$
has unique factorization, which helps, but it has infinitely many units, which causes problems with arguments
that need a number to have finitely many divisors. This is why we will use more powerful machinery
in Section \ref{se:fext} to prove there is a choice of $a_1$ for which Nora wins. We show how to find
$a_1$ in Section \ref{calc}.

It may seem that the last player has the advantage, and this is often the case.
In Sections \ref{se:ZQ} and \ref{se:fext}, we show that if $D=\mathbb Z$ or $D=\mathbb Q$,
or if  $D$ is any subring of a
finite extension of $\mathbb Q$, then the last player wins. 
In Sections~\ref{se:R}, \ref{se:C}, and \ref{se:F}, we find
the exact win conditions for 
the reals, algebraically closed fields, and finite fields.

\section{A Useful Observation}\label{observ}

The following lemma means that we can concentrate most of our efforts on the situation where Nora makes the last play.

\begin{lemma}\label{le:last} 
If Wanda makes the last play,  then she wins.
\end{lemma}

\begin{proof}
If Wanda plays last then, before making the last move, she is looking at a polynomial of the form
$$g(x) + a_ix^i$$
and wants to choose $a_i$ such that the polynomial has a root. She sets $a_i=-g(1)$.
Then 1 is a root. Therefore, she wins, unless $i=0$ or $d$ and $g(1)=0$, in which case she would break the rule that
$a_d a_0\ne 0$. 

In the case $i=0$, since $g(x)$ has only finitely many zeros, she wants to choose $r\in D$ with $g(r)\ne 0$ and then
choose $a_0=-g(r)$. If the cardinality of $D$ is larger than the degree of $g$, namely $d$,
then this is certainly possible.

If $i=d$, let $f_1(x)=x^d f(1/x) = a_0x^d+\cdots +a_d$ be the reversed polynomial.
If the cardinality of $D$ is larger than $d$ then it is possible to choose $a_d\ne 0$ such that
$f_1$ has no zeros, by what we have just proved. Since $a_da_0\ne 0$, we see that
$f$ has no zeros if and only if $f_1$ has no zeros.

Therefore, we are reduced to considering
finite fields $\mathbb F_q$ with $q\le d$. 

If $d\ge 4$, Wanda can arrange that either she or Nora chooses $a_d$ and $a_0$ before the final play. Then, when Wanda chooses the final coefficient,
the problem with $a_da_0\ne 0$ does not arise, so she wins. 

If $d=3$, then Nora is Player I. When Nora chooses $a_0$ or $a_3$, then Wanda chooses the other, setting $a_0=a_3$. When Nora chooses $a_1$ or $a_2$,
Wanda chooses the other, setting $a_1=a_2$. The final polynomial has $-1$ as a root.

If $d=2$, we have to consider only the case $D=\mathbb F_2$ (since $q\le d$ is all that remains). Wanda is Player I and she chooses $a_1=0$.
After Nora chooses $a_0=1$ or $a_2=1$, Wanda chooses the other. Then $1$ is a root, so Wanda wins.
\end{proof}

Note that the situation with finite fields required the additional argument: In the finite field with $p$ elements (where $p$ is prime),
there is no way for Wanda to choose $a_0\ne 0$ so that $x^p-x+a_0$ has a root in this field.
We'll say more about finite fields in Section \ref{se:F}.

\section{Subrings of $\mathbb Q$}\label{se:ZQ}

The following result is an extension of the ideas hinted at in the Introduction.
\begin{theorem}\label{th:ZQ} 
If $\mathbb Z\subseteq D \subseteq \mathbb Q$, the last player wins.
\end{theorem}

\begin{proof}
If  Wanda goes last then she wins, by Lemma~\ref{le:last}.

The following result shows that Nora wins if she plays last.

\begin{proposition}\label{PropQ}
Let $d\ge 2$ and fix $i$ with $0\le i \le d$. Choose $a_j\in \mathbb Q$ for $j\ne i$ subject to the constraints $a_0\ne 0$ and $a_d\ne 0$.
 Then there exists an integer $0\ne a_i\in \mathbb Z$ such that the polynomial
$f(x)=a_dx^d+a_{d-1}x^{d-1}+\cdots + a_1x+a_0$ has no rational roots.
\end{proposition}
\begin{proof}
Recall the {\it Rational Root Theorem}: {\it Let $R$ be a UFD and let $f(x)=a_dx^d+a_{d-1}x^{d-1}+\cdots + a_1x+a_0\in  R[x]$
with  $a_0a_d\ne 0$. If $r$ is in the field of fractions of $R$ and is a root of $f$, then the numerator of $r$ divides $a_0$ and the denominator
of $r$ divides $a_d$.}

We multiply all the coefficients of $f(x)$ by some non-zero integer $N$ to clear denominators.
Therefore, we can assume that all of the coefficients $a_j$ are integers, and we need to find a suitable integer $a_i$
with $N\mid a_i$. Then we can divide by $N$ and obtain the result.

Assume first that $i\ne 0, d$. The Rational Root Theorem (for $R=\mathbb Z$) implies that there is a finite set $S$ of
possibilities for rational roots of $f(x)$, where $S$ is independent of the choice of the integer $a_i$. 
Write $f(x)=g(x)+a_ix^i$.
Let $M=\text{Max}(|g(s)|)$ and let $m=\text{min}(|s|)$,  where $s$ runs through the elements of $S$.
Then $m\ne 0$, because $a_0\ne 0$. If $|a_i|> M/m^i$, then $f(s)\ne 0$ for $s\in S$, so $f$ has no rational
roots. Therefore, we can choose $a_i$  to be any multiple of $N$ satisfying this inequality and obtain the desired 
coefficient.

Now suppose that $i=0$. We then have $f(x)=xh(x)+a_0$, with $0\ne a_0\in N\mathbb Z$ still to be chosen,
and where $h(x)\in \mathbb Z[x]$ has already been determined. If we were asking only for integer roots $x$, things would 
be easy: we could take $a_0$ to be $N$ times a suitable prime $p\nmid N$. This could not be factored as $-xh(x)$ except possibly for finitely
many choices of $x$, namely those where $\pm x$ or $\pm h(x)$ is a divisor of $N$.
But the proposition allows $x$ to be rational, so we need to strengthen the argument.

We choose $a_0=Np$, where $p\nmid N$ is a prime to be specified later and $N$ was used 
above to clear denominators. The Rational Root Theorem implies that
a rational root $r$ of $f(x)$ has the form $x_1/a_d$, where $x_1$ is an integer.
Rewrite $f(r)=0$ as
$$
x_1^d+a_{d-1}x_1^{d-1}+a_{d-2}a_dx_1^{d-2}\cdots + a_d^{d-2}a_1 x_1 = -a_d^{d-1} Np.
$$
This may be written as
$$
x_1 h_1(x_1)= -a_d^{d-1} Np,
$$
where $h_1\in \mathbb Z[x]$. Therefore, either $h_1(x_1)$ or $x_1$ is a (positive or negative) divisor of
$a_d^{d-1}N$. This shows that there a finite set of possibilities for $x_1$, independent of the choice
of $p$. Choose
$a_0=pN$ not equal to any possible value of $-x_1 h_1(x_1)/a_d^{d-1}$. Then the resulting polynomial
$f(x)$ has no rational roots.

Finally, suppose $i=d$, so all coefficients have been chosen except for the leading coefficient.
Let 
$$f_1(x)=x^d f(1/x) = a_d+a_{d-1}x+\cdots +a_0x^d
$$
be the reversed polynomial. By what we just did, we can find $0\ne a_d\in N\mathbb Z$ so that $f_1(x)$
has no rational roots. Since we have $a_0a_d\ne 0$, the roots of $f$ and $f_1$ are non-zero, so
$f(x)$ also has no rational roots.

This completes the proof of Proposition \ref{PropQ}. \end{proof}
This also finishes the proof of Theorem \ref{th:ZQ}. \end{proof}

\smallskip
\noindent
{\bf Remark.} 
The key to the proof of Proposition \ref{PropQ} is that a non-zero integer has only a finite number of divisors.
The proof can be extended to any UFD with a finite number of units and infinitely many irreducibles.

\section{$D$ is a Subring of a Finite Extension of $\mathbb Q$}\label{se:fext}

What happens when $\mathbb Z$ is replaced, for example, by the ring of algebraic integers in a finite extension of $\mathbb Q$? The results of the preceding section can be
generalized to this situation.
We prove the following theorem.

\begin{theorem}\label{th:fext} 
Let $D$ be a subring of a finite extension of $\mathbb Q$.
Whoever plays last wins.
\end{theorem}

\begin{proof}
If  Wanda goes last then, by Lemma~\ref{le:last},
 Wanda wins.

If Nora plays last, she wins by the following result.

\begin{theorem}\label{th:polyfext} Let $K$ be a finite extension of $\mathbb Q$. 
Let $d\ge 2$ and fix $i$ with $0\le i \le d$. Choose $a_j\in K$ for $j\ne i$ subject to the constraints $a_0\ne 0$ and $a_d\ne 0$.
 Then there exists an integer $0\ne a_i\in \mathbb Z$ such that the polynomial
$f(x)=a_dx^d+a_{d-1}x^{d-1}+\cdots + a_1x+a_0$ has no roots in $K$.
\end{theorem}
\begin{proof}
The general proof requires
some ideas from algebraic number theory. In order not to have these obscure the main structure
of the proof, we first give the proof for $K=\mathbb Q$ and then indicate what needs to be modified for the general case. Of course, Proposition \ref{PropQ} already has the result for
$K=\mathbb Q$, but the new ideas are easier to present in this case.

\begin{definition}
Let $S$ be a finite set of primes of $\mathbb Z$.
Define the $S$-units to be 
$$U_S = \left\{ u\frac{a}{b} \,\mid\, a,b\hbox{ are products of primes in $S$ and $u$ 
is a unit of $\mathbb Z$}  \right\}.$$
If $S$ is empty, we take $U_S$ to be the set of units in $\mathbb Z$, namely $\{\pm 1\}$.
Note that if $S$ is non-empty, then $U_S$ is infinite.
\end{definition}

\begin{lemma}\label{le:ufd} 
Suppose $R$ is a UFD and $K$ is the field of fractions of $R$. 
Let $f(x)=a_dx^d+\cdots +a_0\in K[x]$ with $a_da_0\ne 0$ and with $a_i\in R$ for some $i\ne 0, d$.
Then there is a  finite set $\{y_1, \dots, y_m\}$ in $K$ that depends only 
on $\{a_j\,\mid\, j\ne i\}$ (that is, the set is independent 
of the choice of $a_i\in R$) such that if $z\in K$ is a root of $f(x)$, then
$z/y_j$ is a unit of $R$ for some $j$.
\end{lemma}

\begin{proof} Let $A$ be a common denominator of the $a_j$ for $j\ne i$. Then $Af(x)\in R[x]$. If $f(r/s)=0$ 
with $r, s\in R$ and $\gcd(r, s)=1$, then the Rational Root Theorem 
says that $r\mid Aa_0$ and $s\mid Aa_d$. 
Up to multiplication by units of $R$, there are only finitely many divisors of $Aa_0$ and
only finitely many divisors of $Aa_d$. Therefore, up to multiplication by units of $R$, 
there are only finitely many possibilities for $r/s$. 
\end{proof}

Let $N$ be a nonzero integer. Define $\mathbb Z[1/N]$ to be the set of rational numbers that can be expressed
as polynomials in $1/N$ with coefficients in $\mathbb Z$. These are the rational numbers that can be written
as (possibly non-reduced) fractions $a/N^n$ for some integers $a$ and $n$.

\begin{lemma} 
Let $S$ be the set of primes dividing $N$. The units of the ring $\mathbb Z[1/N]$ are the $S$-units 
$U_S$.
\end{lemma}

\begin{proof}
 Let $u\in U_S$. Then the factorization of $u$
contains only primes from $S$. Some of these primes might occur in the factorization with negative exponents,
but there is a power of $N$, say $N^m$, such that the prime factorization of $N^m u$ has only nonnegative
exponents. This means that $N^mu\in \mathbb Z$, so $u\in \mathbb Z[1/N]$.

Since $u\in U_S$, there exists $v\in U_S$ such that $uv=1$, and the same argument shows that $v\in \mathbb Z[1/N]$.
Therefore, the inverse of $u$ is in $\mathbb Z[1/N]$, so $u$ is a unit of $\mathbb Z[1/N]$.

Conversely, suppose $u$ is a unit of $\mathbb Z[1/N]$. Then there exists $v\in \mathbb Z[1/N]$ with $uv=1$.
There exist $m, n$ such that $N^mu\in \mathbb Z$ and $N^n v\in \mathbb Z$. We have $(N^m u) (N^n v)= N^{m+n}$,
and the right side is a product of primes from $S$. Since the numbers on the left are integers,
their factorizations also contain only primes from $S$. Since $N^m$ and $N^m u$ have prime factors
only from $S$, the same is true for $u$. Therefore, $u\in U_S$. 
\end{proof}

We now need to introduce a powerful tool from transcendence theory, the $S$-unit equation.
It is used, for example, to show that there are only finitely many integer solutions
to certain Diophantine equations. See \cite{HiSi}.

Let's start with an example. Let $U_{2,3}$ be the set of rational numbers of the form
$\pm 2^a3^b$, where $a, b$ are integers.
It is possible to have a sum of three elements of $U_{2,3}$ equal to 1. Two such relations are
$$
\frac32 + \frac{-1}3+\frac{-1}6 = 1 \text{ and } 3+(-1) +(-1) = 1
$$
(there are a few more). 
Are there infinitely many such relations? In this form, the answer is Yes:
$$
3^n + (-3^n) + 1 = 1
$$
for all $n\in \mathbb Z$. But this seems like cheating. We are using a zero subsum
to obtain the relations. The $S$-unit Theorem says that if we do not allow zero subsums,
then there are only finitely many relations.
The following is Theorem 3 of~\cite{sunit} for the case $K=\mathbb Q$.

\begin{theorem}\label{th:sunit} 
Let $a_1, \dots, a_n\in \mathbb Q^{\times}$. Suppose that $S$ has cardinality $s$. Then the equation
$$
a_1u_1+\cdots +a_nu_n=1
$$
with $u_i, \dots, u_n\in U_S$ with 
$$
\sum_{i\in I} a_iu_i\ne 0 \text{ for each non-empty subset } I\subseteq \{1, \dots, n\}$$
has at most $(2^{35}n^2)^{n^3s}$ solutions.
\end{theorem}
For the example of $S=\{2, 3\}$ above, the theorem says that there are only finitely many
relations such as $\frac32 + \frac{-1}3+\frac{-1}6 = 1$. The rest must be the ``cheats''
such as $3^n + (-3^n) + 1=1$ with zero subsums.

If  $0\ne a_0\in K$, we can apply the theorem to equations of the form
$$
a_1u_1+\cdots +a_nu_n=-a_0.
$$
Simply divide by $-a_0$ to obtain the form in the theorem.
(We use $-a_0$ to agree with later equations.)

\smallskip
\noindent
{\bf Proof of Theorem~\ref{th:polyfext}.}

We start with the case where $i\ne 0, d$. Write
$$f(x)=a_dx^d+\cdots +a_0,$$ 
where we will pick $a_i\in \mathbb Z$ later.
Let $I_0=\{j\, \mid \, j\ne i, \; 1\le j\le d, \, a_j\ne 0\}$.

Looking forward to the case where $\mathbb Q$ is replaced by $K$, we choose an integer $N>1$ and
work with the ring $\mathbb Z[1/N]$, which is a UFD (its primes are the primes that do not divide $N$).

 From Lemma \ref{le:ufd},  
there is a  finite set $\{y_1, \dots, y_m\}$ in $\mathbb Q$ such that if $z\in \mathbb Q$ is a root of $f(x)$, then
$z/y_j$ is a unit of $\mathbb Z[1/N]$ for some $j$. Choose an integer $M$ whose factorization into primes
includes all primes that occur in the factorizations of $y_1, \dots, y_m$, and such that $M$ is a 
multiple of $N$. Then the units of $\mathbb Z[1/N]$ are contained in the units of $\mathbb Z[1/M]$, so
the set $\{y_1, \dots, y_m\}$ still has the property from Lemma \ref{le:ufd}. But $y_1, \dots, y_m, 1/y_1, \dots, 1/y_m$
are in $\mathbb Z[1/M]$, so we have proved that every root $z$ of $f(x)$ is in $U_S$, where
$S$ is the set of primes dividing $M$.

Let
$$
\text{SOLN}=\{r\in \mathbb Q^{\times} \,\mid\, (\exists n_r\ge 0) [ rg(r) + N^{n_r} r^i  = -a_0] \}.$$
Then $\text{SOLN}\subseteq U_S$. 

If $r\in \text{SOLN}$, then there is a unique corresponding $n_r$ (because $N>1$).
Let 
$$\vec u_r = (r^d, \dots, r^{i+1}, N^{n_r}r^i, r^{i-1}, \dots, r),$$
and let $u_{r,j}=r^j$ if $j\ne i$ and $u_{r,i}=N^{n_r} r^i$.
For each $r\in \text{SOLN}$,
\begin{equation}\label{initialsum}
\sum_{j\in I_0} a_j u_{r,j} + 1\cdot u_{r,i}=-a_0.
\end{equation}

Suppose that SOLN is infinite.
As $r$ varies through infinitely many values, so do all the entries of $\vec u_r$, except possibly for
the entry $N^{n_r}r^i$. Since $0\ne a_d\in I_0$, Equation (\ref{initialsum}) yields infinitely many relations.
Theorem \ref{th:sunit} implies that, for each $r$ outside of a finite set, there is a subsum of Equation (\ref{initialsum}) that is 0.
Since there are only finitely many subsets of $I_0$, there is a subset $\emptyset \ne I\subset I_0$
such that either
$$\sum_{j\in I} a_j u_{r,j} = 0 \text{ for infinitely many $r$}
$$
or
$$
 \sum_{j\in I} a_j u_{j,r} + 1\cdot N^{n_r} r^i=0  \text{ for infinitely many $r$}.
$$
The first possibility is impossible, because the nonzero polynomial
$$
\sum_{j\in I} a_j X^j
$$
cannot have infinitely many zeros.  Subtracting the second possibility from Equation (\ref{initialsum}) implies that 
$$
\sum_{j\not\in I} a_j X^j= -a_0
$$
has infinitely many solutions. Since $a_0\ne 0$, this is a non-trivial polynomial relation, so we again have a contradiction.

Therefore, SOLN is finite. Each $r$ has a unique 
$n_r$. Let $n$ be greater
than the largest $n_r$. Then $a_i=N^n$ yields a polynomial with no roots in $\mathbb Q$.
This completes the case $i\ne 0, d$.

Now suppose that $i=d$.  We take $a_d=N^n$ for some
yet-to-be-determined $n$. Then $a_d$ is a unit of $\mathbb Z[1/N]$, so it does not affect divisibility, 
and the proof of Lemma \ref{le:ufd}
yields a set $\{y_1, \dots, y_m\}$ as before.
The proof now proceeds as previously. If we suppose that there are infinitely many distinct relations as $r$ varies,
then we obtain a contradiction and deduce that SOLN is finite.

However, there is the possibility that the infinitely many $r\in \text{SOLN}$ yield only finitely many relations in Equation (\ref{initialsum}).
If this happens, then $a_{d-1}=a_{d-2}=\cdots = a_1=0$, since the corresponding components of $\vec u_r$ take on infinitely many
values as $r$ runs through infinitely many elements of $\text{SOLN}$.  Therefore, the polynomial is
$f(x)=a_dx^d+a_0$, with $a_0$ fixed and $a_d$ to be determined. Since $K$ is a finite extension
of $\mathbb Q$, there are infinitely many primes $p$ such that $x^d-p$ has no roots in $K$
(any $p$ that does not ramify in $K/\mathbb Q$ suffices). Choose $a_d=-N^{kd}a_0^{d+1}p$,
where $k$ is chosen large enough to make $a_d\in \mathbb Z$. Then
$$
-f(x)/a_0 =p(N^ka_0x)^d-1,
$$
which has no roots in $K$.

Finally, suppose $i=0$. Let $f_1(x)=x^d f(1/x) = a_0x^d+ \cdots + a_d$ be 
the reversed polynomial. Since $a_0a_d\ne 0$, $f$ has a zero in $K$ if and only if $f_1$
has a zero in $K$. The above shows that there exists $a_0\ne 0$ such that $f_1$ has no roots in
$K$, as desired.

This completes the proof when $K=\mathbb Q$.

\noindent
{\bf The General Case}

We now indicate what needs to be done when $K$ is a finite extension of $\mathbb Q$. Let $A$ be the
ring of algebraic integers in  $K$. The following key step allows us to use UFD's during the rest of the proof.

\begin{lemma}\label{le:local}
Let $K$ be a finite extension of $\mathbb Q$ and let $A$ be the ring of algebraic integers in $K$.
There exists $0\ne N\in \mathbb Z$ such that $A[1/N]$ is a UFD.
\end{lemma}

\begin{proof}
If $J$ is an ideal of $A[1/N]$, then there is an ideal $I$ of $A$ such that $IA[1/N]=J$ (this is a standard fact
about localization of rings).

The classical result on the finiteness of the class number of $A$ says that there is a set $\{I_1, \dots, I_h\}$
of ideals of $A$ with the following property: If $I$ is a nonzero ideal of $A$, then there
are nonzero $r, s\in A$ and $i\le h$ such that $rI=sI_i$. 

Choose $0\ne N\in I_1\cap I_2\cap \cdots \cap I_h$ (such integers exist; for example, the index of this intersection
of ideals in $A$). Let $J$ be a nonzero ideal of $A[1/N]$, and choose $I$ so that $IA[1/N]=J$.
Let $r, s\in A$ and $i\le h$ be such that $rI=sI_i$. Then 
$$
rJ = rIA[1/N]=sI_iA[1/N] = sA[1/N].
$$
The last equality is because $N\in I_i$, so $1=N\cdot (1/N) \in I_iA[1/N]$, which means $I_iA[1/N]=A[1/N]$.
Therefore, $rJ$ is a principal ideal of $A[1/N]$, which implies that $J$ is principal
(this last deduction requires a little machinery, for example Dedekind domains).

Since $J$ was an arbitrary ideal of $A[1/N]$, we have proved that $A[1/N]$ is a PID, therefore a UFD.
\end{proof}

If $S$ is a finite set of prime ideals of $A$, define $U_S$ to be the nonzero elements $u\in K$ such that
the prime ideal factorization of the fractional ideal of $A$ generated by $u$ contains only primes from $S$. 
If $0\ne N\in \mathbb Z$ and $S$ is the set of prime ideals of $A$ dividing $N$, then the units of $A[1/N]$ are exactly
$U_S$.

The rest of the proof is the same as before, including
Theorem \ref{th:sunit},
with $\mathbb Z$ replaced by $A$ and $\mathbb Q$ replaced by $K$.

This completes the proof of Theorem \ref{th:polyfext}. \end{proof}

This also completes the proof of Theorem \ref{th:fext}. \end{proof}

\section{Does the polynomial have a root?}\label{calc}

The proof of Theorem \ref{th:fext} shows that there exists a win for the last player.
If the last player is Wanda, she finds the desired coefficient easily (see Lemma \ref{le:last}).
But suppose the last player is Nora. Two questions arise: 
\begin{enumerate}
\item How does Nora find the coefficient?
\item Once she finds it, how does she verify that there are no roots in the field?
\end{enumerate}
The cases $i=0$ and $i=d$ can be treated by slight variations of what we do in this section, so we restrict to $i\ne 0, d$.
The proof of Theorem \ref{th:polyfext} shows that (when $i\ne 0, d$) there are only finitely many
$n$ such that $a_i=N^n$ yields a polynomial with a root. Therefore,
Nora can try $a_i=N^n$ for $n=1, 2, 3, \dots$ until she finds the desired polynomial.
In fact, Theorem \ref{th:sunit}  implies a bound on how many $n$ will not yield the desired polynomial,
hence a bound on how far Nora needs to look.

But now, suppose Nora has found what she believes is a good coefficient.
How does she verify that there are no roots in the field? Even more important,
how does she prove to Wanda that there are no roots in the field?

The book \cite[Section 3.6.2]{Cohen} shows how to answer this question. Let $K=\mathbb Q(\theta)$ be an extension of
$\mathbb Q$ of degree $n$, and let $\sigma_1, \dots, \sigma_n$ be the embeddings of $K$ into $\mathbb C$ (these are essentially the Galois group
if $K/\mathbb Q$ is Galois). If $A(x)\in K[x]$, let 
$$
N(A(x))=\prod_{i=1}^n \sigma_i(A(x)),
$$
where $\sigma_i(A(x))$ denotes $\sigma_i$ applied to the coefficients of $A(x)$. Then $N(A(x))\in \mathbb Q[x]$. If $A(x)$ is squarefree,
there is an explicit finite set of rational numbers $k$ such that $N(A(x-k\theta))$ is squarefree when $k$ is a rational number not in this set
(see \cite[Lemma 3.6.2] {Cohen}). 
\begin{theorem} (\cite[Lemma 3.6.3]{Cohen}) Assume that both $A(x)\in K[x]$ and $N(A(x))\in\mathbb Q[x]$ are squarefree. Let $N(A(x))=\prod_j N_j(x)$
be the factorization of $N(A(x))$ into irreducibles in $\mathbb Q[x]$. Then
$$
A(x)=\prod_{j} \gcd(A(x), N_j(x))
$$
is the factorization of $A(x)$ into irreducibles in $K[x]$.
\end{theorem}
The theorem allows us to determine whether $A(x)$ has a linear factor in $K[x]$, which happens if and only if $A(x)$ has a root in $K$.

If a polynomial $A(x)$  is not squarefree, we can eventually reduce to the squarefree situation by writing 
$$
A(x)=\gcd(A, A')\times\left(A(x)/\gcd(A, A')\right),
$$
where $A'$ is the derivative, and treating each factor separately. 
If $N(A(x))$ is not squarefree, we can translate by a suitable $k\theta$, apply the theorem, and then translate back.

Let's consider an example. Start with the polynomial
$$
f(x)= x^3+(\sqrt{2}-3)x^2+a_1x-4(1+\sqrt{2}).
$$
Nora wants to choose $a_1$ so that $f(x)$ has no roots in $\mathbb Q(\sqrt{2})$.

She first tries $a_1=2$ and computes the product of the Galois conjugates of $f$:
\begin{align*}
&N(f(x))\\
&= \left(x^3+(\sqrt{2}-3)x^2+2x-4(1+\sqrt{2})\right)
\left(x^3+(\sqrt{2}-3)x^2+2x-4(1+\sqrt{2})\right)\\
&= x^6 - 6x^5 + 11x^4 - 20x^3 + 44x^2 - 16x - 16\\
&= (x^2-2x-1)(x^4-4x^3+4x^2-16x+16).
\end{align*}
Computing $\gcd(f(x),\, x^2- 2x-1)$ by the Euclidean algorithm
yields the linear polynomial $(1+2\sqrt{2})(x-1-\sqrt{2})$. This means that $1+\sqrt{2}$ is a root in $\mathbb Q(\sqrt{2})$.

When she tries $a_1=4$, she computes a new $N(f(x))$, obtaining
$$
N(f((x)) = (x-2)^2(x^4-2x^3+3x^2-12x-4).
$$
The squared factor corresponds to the fact that $x=2$ is a root of $f(x)$, but let's ignore this and try to obtain a squarefree $N(f(x))$ in order to use the theorem. Compute
$$
f_1(x)=f(x-\sqrt{2})=x^3-(3+2\sqrt{2})x^2+6(1+\sqrt{2})x-(10+8\sqrt{2}).
$$
Then
$$
N(f_1(x))= (x^2 - 4x + 2)(x^4 - 2x^3 + 3x^2 + 8x - 14).
$$
We have
$$
\gcd(f_1(x), x^2 - 4x + 2) = (8-2\sqrt{2})(x-2-\sqrt{2}).
$$
Therefore, $f(2)=f_1(2+\sqrt{2})=0$, so $f(x)$ has a zero.

Now Nora tries $a_1=8$. The new $N(f(x))$ is
$$
x^6 - 6x^5 + 23x^4 - 56x^3 + 104x^2 - 64x - 16,
$$
which is irreducible in $\mathbb Q[x]$. 
We have 
$$
\gcd(N(f(x)), f(x))= f(x).
$$
Therefore, $f(x)$ does not have a linear factor
in $K[x]$. Nora chooses $a_1=8$ and wins.

\section{$D=\mathbb R$}\label{se:R}

Although one might suspect that the case of $D=\mathbb R$ is easy,
it turns out to be interesting and it uses a strategy that encompasses more than the last turn.

\begin{theorem}\label{th:R}~ Let $D=\mathbb R$. 
\begin{enumerate}
\item If $d=2$, Player I wins.
\item If $d\ne 2$, then Wanda wins.
\end{enumerate}
\end{theorem}

\begin{proof}
If $d$ is odd, then Wanda always wins because odd degree polynomials always have real roots.

Suppose $d=2$. If Player I is Wanda, she is also the last player, so she wins.
If Player I is Nora, then Nora starts by choosing $a_1=0$. Wanda  then chooses $a_0$ or $a_2$ and
Nora finally chooses $a_0=a_2$ and wins.

Henceforth, assume that $d\ge 4$ is even. 
If Nora chooses either $a_0$ or $a_d$, then Wanda chooses the other of these and arranges that
$a_0$ and $a_d$ have opposite signs. Then the polynomial takes opposite signs for $x=0$
and for large positive $x$, so Wanda wins. Therefore, Nora's only hope is to avoid $a_0$ and $a_d$
and try to force Wanda to choose one of them before Nora does. 

The number of coefficients to be chosen is $d+1$, which is odd. If Wanda is Player I, she
also is the last player, so she wins by Lemma \ref{le:last}.
If Nora is Player I, then she is also the last to play, so she has some hope.
But Wanda does the following after Nora makes the first choice (which is not $a_0$ or $a_d$, by the above). On Wanda's first play, she chooses $a_d>0$. Nora must respond by
choosing $a_0>0$. Otherwise, Wanda will choose $a_0<0$ and win the game, as described above.
They are now in the situation where coefficients $a_d$, $a_0$, and $a_j$, for some $j\ne 0, d$, have been
chosen and it is Wanda's turn to choose.

Wanda's strategy is, until her last move, to set some coefficient of even degree equal to 0.
Since $d\ge 4$ and $d$ is even, before Wanda's last move there are two coefficients left to set.
Because of the strategy Wanda uses, it is impossible for both coefficients left to be of even degree.
There are two cases.

\noindent
{\bf Case a:} Of the last two coefficients, one is of even degree and one is of odd degree.
Let the entire polynomial be
$$f(x)=g(x) + a_{2i}x^{2i} + a_{2j+1}x^{2j+1}$$
where $a_{2i}$ and $a_{2j+1}$ have not been determined yet.
It will not matter whether $2i<2j+1$ or $2i>2j+1$.
Note that, since $a_0$ was already set, $2i, 2j+1\ge 1$.

Wanda sets $a_{2i}$ to a value $-A$ that we will determine later.
Nora will respond by setting $a_{2j+1}$ to a value $B$. 
Wanda picks $A$ so that, no matter what $B$
Nora picks, there will be a root.

Since $\lim_{x\to\infty} f(x)=\infty$, we need show only that there is some value $x$
such that $f(x)\le 0$.

Note that
$$f(1)=g(1) - A + B.$$
Hence Nora needs to make
$$g(1)-A+B > 0, \quad B> -g(1)+A.$$
Note also that 
$$f(-1)=g(-1) -A-B.$$
Hence Nora needs to make
$$g(-1) -A-B>0, \quad B < g(-1)-A.$$
Putting these together, Nora needs to pick a $B$ such that
$$-g(1)+A< B < g(-1)-A.$$
If Wanda can find an $A$ such that
$$g(-1)-A < -g(1)+A$$
then Nora cannot pick a winning $B$.
Hence Wanda's winning move is to pick any $A$ with
$$A > \frac{g(1)+g(-1)}{2}.$$

\noindent
{\bf Case b:} The last two coefficients are both of odd degree. Let $g(x)$ be the part of the polynomial
that is already set. Let the entire polynomial be
$$f(x)=g(x) + a_{2i+1}x^{2i+1} + a_{2j+1}x^{2j+1},$$
where $a_{2i+1}$ and $a_{2j+1}$ have not been determined yet. We assume that $i>j$.
Note that $2i+1,2j+1\ge 1$.

Wanda plays by setting $a_{2i+1}$ to a value $A$, which she picks so that, no matter what $B$
Nora picks, there will be a root.

Since $\lim_{x\to\infty} f(x)=\infty$, we need show only that there is some value $x$
such that $f(x)\le 0$. 
Note that
$$f(1)=g(1) + A + B.$$
Hence Nora needs to make
\begin{gather*}
g(1)+A+B > 0,\\
B> -g(1)-A.
\end{gather*}
Note also that 
$$f(-2)=g(-2) -2^{2i+1}A-2^{2j+1}B.$$
Hence Nora needs to make
\begin{gather*}
g(-2) -2^{2i+1}A-2^{2j+1}B>0\\
2^{2j+1}B < g(-2) - 2^{2i+1}A.
\end{gather*}
Putting these together, Nora needs to pick a $B$ such that
$$-g(1)-A < B < g(-2)2^{-2j-1}-2^{2i-2j}A.$$
If Wanda can find an $A$ such that
$$g(-2)2^{-2j-1}-2^{2i-2j}A < -g(1)-A$$
then Nora cannot pick a winning $B$.
Hence Wanda's winning move is to pick any $A$ with
$$A(2^{2i-2j}-1) > g(1)+g(-2)2^{-2j-1}.$$
\end{proof}

\section{$D$ an Algebraically Closed Field }\label{se:C}

The case $D=\mathbb C$, or any other algebraically closed field, is of course trivial.

\begin{theorem}
Suppose $D$ is an algebraically closed field. Then
 Wanda wins.
\end{theorem}

\begin{proof}
The polynomial will be nonconstant of degree $\ge 1$, so it has a root
since $D$ is algebraically closed.
\end{proof}

\section{$D$ a Finite Field}\label{se:F}

The case where $D=\mathbb F_q$, the finite field with $q$ elements, brings in some new ideas.

\begin{theorem}\label{th:F} Let $D=\mathbb F_q$.
\begin{enumerate}
\item  If $d\ge 4$ or $d=2$, then
whoever plays last wins. 
\item If $d=3$ and the characteristic of $\mathbb F_q$ is 3, then Wanda wins.
\item If $d=3$ and the characteristic of $\mathbb F_q$ is not 3, then the last player wins.
\end{enumerate}
\end{theorem}

\begin{proof}
If  Wanda goes last, she wins,  by Lemma~\ref{le:last}.

Henceforth, assume  Nora goes last.
On the last turn, Nora 
wants to choose the remaining coefficient $a_i$ so that the polynomial has no roots $\mathbb F_q$.

Suppose $d\ge 4$. Then Nora can arrange that either she or Wanda chooses values for $a_d$ and $a_0$ before
the last turn. So we may assume that $i\ne 0, d$.
Since $a_0\ne 0$, we cannot have 0 as a root. For each $b\in \mathbb F_q^{\times}$, 
there is exactly one value of $a_i$ for which $b$ is a root. 
As $b$ runs through $\mathbb F_q^{\times}$, 
Nora eliminates at most $q-1$ possible values of $a_i$, so she can pick some $a_i$ so that each 
$b\in \mathbb F_q^{\times}$ is not a root. The resulting polynomial then has no roots.

When $d=2$, Nora starts by choosing $a_1=1$. If Wanda chooses $a_2\ne 0$, then
both $0$ and $-1/a_2$ are roots of $a_2x^2+x$. Therefore, the image of the map
$x\mapsto a_2x^2+x$ from $\mathbb F_q$ to $\mathbb F_q$ has at most $q-2$ nonzero elements. This means that there is a $-a_0\ne 0$
not in the image. This choice of $a_0$ makes $f(x)$ have no roots. Nora wins.
If, instead, Wanda chooses $a_0\ne 0$ on her first play, then Nora must choose $a_2$ such that
$a_2x^2+x+a_0$ has no roots. Each of the $q-2$ values of $x\ne 0, -a_0$ eliminates one value of $a_2$, so at least one
nonzero $a_2$ remains.  Since $x=0$ and $x=-a_0$ are not roots of the resulting polynomial, there are no roots,
and Nora wins. 

We now need a quick interlude on permutation polynomials.
A polynomial $g(x)\in \mathbb F_q[x]$ is called a {\it permutation polynomial}
if the map $x\mapsto g(x)$ gives a permutation of $\mathbb F_q$.
Dickson \cite{Di} classified all permutation polynomials of degree at most 5.
We need only the result for degree 3:

{\it
The only permutation polynomials of degree 3 are of the form $$g(x)=a_3h(x+b)+c,$$ where
$a_3, b, c\in \mathbb F_q$ with $a_3\ne 0$ and $h(x)$ is one of the following:
\begin{enumerate}
\item $x^3$, with $q\not\equiv 1\pmod 3$.
\item $x^3-ax$, with $q\equiv 0\pmod 3$ and $a$ not a square in $\mathbb F_q$.
\end{enumerate} }
Note that the first case can be written as
$$
g(x)=a_3(x^3+3bx^2+3b^2x) + a_3b^3+c.
$$
Therefore, if $g(x)$ is a permutation polynomial of degree 3, and the characteristic is not
3, then the coefficients of $x^2$ and $x$ are either both zero or both nonzero.

We can now treat the case $d=3$. 
If the remaining coefficient is $a_i$ with $i\ne 0, 3$, then
the earlier argument for $d\ge 4$ shows that Nora wins. Therefore, it remains to 
consider the cases where $i=0$ and $i=3$.
 
Write $f(x) = g(x)+a_0$, where $g$ is a polynomial
with $g(0)=0$ and $a_0$ is to be determined.
If every choice of $a_0$ yields a polynomial $f(x)$ with a root in $\mathbb F_q$, then $g(x)$ gives a surjective
map from $\mathbb F_q$ to $\mathbb F_q$, so
$g(x)$ is a permutation polynomial. 

Suppose that the characteristic of $\mathbb F_q$ is 3. Wanda chooses $a_2=0$, so Nora must choose
a coefficient of $a_3x^3+a_1x+a_0$. 

If Nora chooses $a_1=0$, then Wanda chooses $a_3=1$ and each choice
of $a_0$ yields a root. If Nora chooses $a_1\ne 0$, then Wanda chooses $a_3$ so that $-a_1/a_3$ is not a
square in $\mathbb F_q$. Then $a_3x^3 + a_1x$ is a permutation polynomial, so Nora cannot choose
a good $a_0$ and she loses. Therefore, Nora should not choose $a_1$ on her first turn.

If Nora chooses $a_3\ne 0$, then Wanda chooses $a_1=0$, thus yielding the permutation
polynomial $a_3x^3$ for $g(x)$. Thus Nora cannot choose a good $a_0$, hence she loses.

If Nora chooses $a_0\ne 0$, then Wanda chooses $a_1=0$ and again Nora loses.

Therefore, when $d=3$ and the characteristic is 3, Nora loses.

Finally, assume that $d=3$ and the characteristic of $\mathbb F_q$ is not 3. If Wanda 
chooses one of $a_0$, $a_3$, then Nora chooses the other and the argument 
at the beginning
of the proof shows that Nora wins. Therefore, Wanda must choose $a_1$ or $a_2$. 
Nora then chooses the other of $a_1$ and $a_2$ and arranges that exactly
one of $a_1$ and $a_2$ is 0. If Wanda then chooses $a_3$, the resulting $g(x)$ cannot be a permutation polynomial,
so Nora can win. If, instead,  Wanda chooses $a_0$ on her second term, 
then Nora must choose $a_3$. Let $f_1(x)=x^3f(1/x)$
be the reversed polynomial. Since $a_0a_3\ne 0$, we see that $f_1$ has no zeros if and only if $f$ has no zeros.
Because exactly one of $a_1$ and $a_2$ is 0, the polynomial $a_0x^3+a_1x^2+a_2x$ cannot be a permutation
polynomial. Therefore, it is possible to choose $a_3\ne 0$ such that $f_1$ has no zeros. Therefore,
$f$ has no zeros and Nora wins.

This completes the proof of Theorem \ref{th:F}.
\end{proof}

\section{Open Problems}\label{se:open}

\noindent
{\bf 1.} We would like to see an elementary proof of Theorem~\ref{th:fext}.

\noindent
{\bf 2.} The bound $(2^{35}n^2)^{n^3s}$ in Theorem \ref{th:sunit} applies to a very general
situation. Can it be substantially improved in the special situation in which the theorem is applied in the proof of Theorem \ref{th:fext}?

\noindent
{\bf 2.} There are two variants of the game that we leave as open fields of study.
\begin{enumerate}
\item
Let $D$ be a ring rather than an integral domain. For example, what happens if $D=\mathbb Z_n$
where $n$ is not prime?
\item
The parameters include two domains $D_1$ and $D_2$ where $D_1,D_2$ are both subsets of the same
larger domain.
The players pick coefficients from $D_1$; however, the root can be in $D_2$.
\end{enumerate}

A particularly interesting case of $(D_1, D_2)$ is the following. Let $D_1=\mathbb Q$ and let $D_2$ be
the compositum of all Galois extensions of $\mathbb Q$ with solvable Galois group.
Hence we are asking if one of the players can force the polynomial to
have a solution in radicals. For $1\le d\le 4$, Wanda 
wins because of the quadratic, cubic, and quartic formulas.
We would like to know what happens when $d\ge 5$.
These could be called {\it Galois Games}; however, that name
has already been taken by a game involving bad duelists~\cite{galoisgames}.

\end{document}